\documentclass[12pt]{article}
\usepackage{graphicx}
\usepackage{amsmath}

\newif\ifpdf
\newtheorem{theorem}{Theorem}

\newtheorem{corollary}[theorem]{Corollary}

\newtheorem{lemma}[theorem]{Lemma}

\newtheorem{problem}[theorem]{Problem}

\def\vertex(#1){\put(#1){\circle*{1.8}}}
\def\lab(#1)#2{\put(#1){\makebox(0,0)[c]{#2}}}
\newenvironment{proof}[1][Proof]{\textbf{#1.} }{\ \rule{0.5em}{0.5em}}
\input{tcilatex}

\begin{document}

\title{Results on the intersection graphs of subspaces of a vector space}
\author{N. Jafari Rad and S. H. Jafari\\[3ex]
{\normalsize Department of Mathematics,} \\
{\normalsize Shahrood University of Technology,}\\
{\normalsize Shahrood, Iran}\\
{\normalsize n.jafarirad@shahroodut.ac.ir}}
\date{}
\maketitle

\begin{abstract}
For a vector space $V$ the \emph{intersection graph of subspaces}
of $V$, denoted by $G(V )$, is the graph whose vertices are in a
one-to-one correspondence with proper nontrivial subspaces of $V$
and two distinct vertices are adjacent if and only if the
corresponding subspaces of $V$ have a nontrivial (nonzero)
intersection. In this paper, we study the clique number, the
chromatic number, the domination number and the independence
number of the intersection graphs of subspaces of a vector space.
\end{abstract}


\section{ Introduction}
For graph theory terminology in general we follow ~\cite{w}.
Specifically, let $G =(V, E)$ be a graph with vertex set $V$ of
order~$n$ and edge set $E$. If $S$ is a subset of $V(G)$, then we
denote by $G[S]$ the subgraph of $G$ induced by $S$. A set of
vertices $S$ in $G$ is a \emph{dominating set}, if $N[S]=V(G)$.
The \textit{domination number}, $\gamma (G)$, of $G$ is the
minimum cardinality of a dominating set of $G$. A set of vertices
$S$ in $G$ is an \emph{independent set}, if $G[S]$ has no edge.
The \textit{independence number}, $\alpha (G)$, of $G$ is the
maximum cardinality of an independent set of $G$. The
\textit{clique number} of a graph $G$, written $w(G)$, is the
maximum size of a set of pair-wise adjacent vertices of $G$. A
function $f$ defined on $V(G)$ is a proper \emph{vertex coloring}
if $f(v)\neq f(v)$ for any pair of adjacent vertices $u,v$. The
(vertex) \emph{chromatic number}, $\chi(G)$, of $G$ if the
minimum $k$ such that there is a proper vertex coloring $f$ on
$G$ with $|f(V(G))|=k$.\\

Let $F=\{S_i:i\in I\}$ be an arbitrary family of sets. The
\textit{intersection graph} $G(F)$ is the one-dimensional skeleton
of the nerve of $F$ , i.e., $G(F)$ is the graph whose vertices are
$S_i$, $i\in I$ and in which the vertices $S_i$ and $S_j$
($i,j\in I$) are adjacent if and only if $S_i\neq S_j$ and
$S_i\cap S_j\neq \emptyset$ \cite{s-m}.\\

The study of algebraic structures using the properties of graphs
has become an exciting research topic in the last few decades
years, leading to many fascinating results and questions. It is
interesting to study the intersection graphs $G(F)$ when the
members of $F$ have an algebraic structure. For references of
intersection graphs of algebraic structures see for example
\cite{b,cgms,cp,z}.\\

Intersection graphs of subspaces of a vector space are studied by
Jafari Rad and Jafari in \cite{jj0,jj}. For a vector space $V$
the \emph{intersection graph of subspaces} of $V$, denoted by $G(V
)$, is the graph whose vertices are in a one-to-one
correspondence with proper nontrivial subspaces of $V$ and two
distinct vertices are adjacent if and only if the corresponding
subspaces of $V$ have a nontrivial (nonzero) intersection.
Clearly the set of vertices is empty if $dim(V ) = 1$. Jafari Rad
and Jafari characterized all vector spaces whose intersection
graphs are connected, bipartite, complete, Eulerian, or planar.\\

In this paper, we continue the study of the intersection graph of
subspaces of a vector space. We study the chromatic number, the
clique number, the domination number, and the independence number
in the intersection graph of subspaces of a vector space.
Throughout this paper $V$ is a vector space with $\dim(V)=n$ on a
finite field $F$ with $|F|=q$. We also denote by $0$ the zero
subspace of a vector space.

\section{Known results}

Let $F$ be a finite field with $|F|=q$, and let $V$ be an
$n$-dimensional vector space over $F$. For integer $t\in
\{1,2,...,n\}$, the number of $t$-dimensional subspaces of $V$ is
given in \cite{frankl} by $\bigg[
\begin{array}{c}
n \\
t
\end{array}
\bigg]_{q} =\prod_{0 \leq i <t} \frac{q^{n-i}-1}{q^{t-i}-1}.$ We
suppose that $\bigg[
\begin{array}{c}
n \\
0
\end{array}
\bigg]_{q}=1$ and $\bigg[
\begin{array}{c}
n \\
t
\end{array}
\bigg]_{q}=0$ if $t\not \in \{0,1,2,...,n\}$. Note that $\bigg[
\begin{array}{c}
n \\
t
\end{array}
\bigg]_{q}=\bigg[
\begin{array}{c}
n \\
n-t
\end{array}
\bigg]_{q}$ for any $t\in\{0,1,...,n\}$.

\begin{lemma}[Jafari Rad and Jafari \cite{jj}]\label{hh}
If $\dim(W)=m$, then $$|\{W^{\prime}:\dim(W^{\prime})=t, W\cap
W^{\prime}=0\}|=q^{mt}\bigg[
\begin{array}{c}
n-m \\
t
\end{array}
\bigg]_{q}.$$
\end{lemma}

\begin{theorem}[Jafari Rad and Jafari \cite{jj}]\label{}
If $\dim(W)=m$, then $$\deg(W)=\sum_{t=0}^{n} \bigg[
\begin{array}{c}
n \\
t
\end{array}
\bigg]_{q}-\sum_{t=0}^{n-m} q^{mt}\bigg[
\begin{array}{c}
n-m \\
t
\end{array}
\bigg]_{q} -2.$$
\end{theorem}

\begin{theorem}[Jafari Rad and Jafari \cite{jj}]\label{j1}
Let $V$ be a vector space. Then $G(V)$ is connected if and only
if $\dim(V)\geq 3$.
\end{theorem}

We next state some known results of graph theory.

\begin{theorem}[Hall's Marriage Theorem]\label{g1}
For $k>0$, every $k$-regular bipartite graph has a perfect
matching.
\end{theorem}

\begin{lemma}[\cite{w}]\label{g2}
For every graph $G$, $\chi(G)\geq w(G)$.
\end{lemma}

\begin{theorem}[Brooks, \cite{br}]\label{g3}
If $G$ is a connected graph other than a complete graph or an odd
cycle, then $\chi(G)\leq \Delta(G)$.
\end{theorem}

\section{New results}

\begin{theorem}\label{gh}
If $n$ is odd, then
$$w(G(V))=\sum_{i=1}^{\lfloor\frac{n}{2}\rfloor} \bigg[
\begin{array}{c}
n \\
i
\end{array}
\bigg]_{q}$$
\end{theorem}

\begin{proof} First notice that if $W_1,W_2$ are two subspace
of $V$ with $dim(W_i)\geq \frac{n}{2}$ for $i=1,2$, then $W_1\cap
W_2\neq 0$. This means that $$w(G(V))\geq |\{W:W\leq V, dim(W)\geq
\frac{n}{2}\}|=\sum_{i=1}^{\lfloor\frac{n}{2}\rfloor} \bigg[
\begin{array}{c}
n \\
i
\end{array}
\bigg]_{q}.$$ Now we show that $w(G(V))\leq |\{W:W\leq V,
dim(W)\geq \frac{n}{2}\}|$. Let $A_t$ be the set of all
$t$-dimensional subspaces of $V$ with $1\leq t\leq n-1$. For
$1\leq t\leq \frac{n}{2}$, let $H_t$ be the induced bipartite
subgraph of $G(V)$ with partite sets $A_t$ and $A_{n-t}$. Let
$\overline{H} _t$ be the complement of $H_t$ in
$K_{|A_t|,|A_{n-t}|}$. By Lemma \ref{hh}, $\overline{H} _t$ is a
$k$-regular graph with $k=q^{t(n-t)}$. By Theorem \ref{g1},
$\overline{H} _t$ has a perfect matching $M_t$. Now we consider a
proper vertex coloring $f$ for $G(V)$ as follows. Let $W\leq V$
and assume that $n$ is odd. If $dim(W)<\frac{n}{2}$, then
$f(W)=W$, and if $dim(W)=n-t>\frac{n}{2}$, then $f(W)=f(W^*)$
where $dim(W^*)=t<\frac{n}{2}$ and $WW^*\in M_t$. We deduce that
$\chi(G(V))\leq |\{W:W\leq V, dim(W)\geq \frac{n}{2}\}|.$ By
Lemma \ref{g2}, $w(G(V))\leq \chi(G(V))\leq |\{W:W\leq V,
dim(W)\geq \frac{n}{2}\}|$. This completes the proof.
\end{proof}

\begin{corollary}\label{4.2}
If $n$ is odd, then
$\chi(G(V))=\sum_{i=1}^{\lfloor\frac{n}{2}\rfloor} \bigg[
\begin{array}{c}
n \\
i
\end{array}
\bigg]_{q}$.
\end{corollary}

\begin{theorem}\label{4.3}
If $n$ is even then $$\sum_{i=1}^{\frac{n}{2}-1} \bigg[
\begin{array}{c}
n \\
i
\end{array}
\bigg]_{q}+\bigg[
\begin{array}{c}
n-1 \\
\frac{n-2}{2}
\end{array}
\bigg]_{q}\leq w(G(V))\leq \sum_{i=1}^{\frac{n}{2}-1} \bigg[
\begin{array}{c}
n \\
i
\end{array}
\bigg]_{q}+\bigg[
\begin{array}{c}
n \\
\frac{n}{2}
\end{array}
\bigg]_{q}-q^{\frac{n^2}{4}}-1.$$
\end{theorem}

\begin{proof} First notice that if $W_1,W_2$ are two subspace
of $V$ with $dim(W_i)> \frac{n}{2}$ for $i=1,2$, then $W_1\cap
W_2\neq 0$. Also if $W_1,W_2$ are two subspace of $V$ with
$dim(W_1)= \frac{n}{2}$ and $dim(W_2)>\frac{n}{2}$, then $W_1\cap
W_2\neq 0$. But for a $1$-dimensional subspace $V_1$ of $V$, any
subspace of $\frac{V}{V_1}$ of dimension $\frac{n-2}{2}$ is in
the form $\frac{W}{V_1}$, where $dim(W)=\frac{n}{2}$. So there
are at least $\bigg[
\begin{array}{c}
n-1 \\
\frac{n-2}{2}
\end{array}
\bigg]_{q}$ subspaces of dimension $\frac{n}{2}$ which contain
$V_1$. We deduce that

\begin{eqnarray*}
w(G(V)) & \geq & |\{W:W\leq V, dim(W)> \frac{n}{2}\}|+\bigg[
\begin{array}{c}
n-1 \\
\frac{n-2}{2}
\end{array}
\bigg]_{q}\\
 & \geq & \sum_{i=1}^{\frac{n}{2}-1} \bigg[
\begin{array}{c}
n \\
i
\end{array}
\bigg]_{q}+\bigg[
\begin{array}{c}
n-1 \\
\frac{n-2}{2}
\end{array}
\bigg]_{q}.
\end{eqnarray*}

Let $A_t$ be the set of all $t$-dimensional subspaces of $V$ with
$1 \leq t\leq n-1$. Let $G_1=G[A_{\frac{n}{2}}]$ and $G_2=G-G_1$.
As in the proofs of Theorem \ref{gh} and Corollary \ref{4.2}, we
obtain that $G_2$ is $k$-regular with $k=q^{\frac{n}{2}
(n-\frac{n}{2})}=q^{\frac{n^2}{4}}$ and
$\chi(G_2)=\sum_{i=1}^{\frac{n}{2}-1} \bigg[
\begin{array}{c}
n \\
i
\end{array}
\bigg]_{q}$. Since $G_1$ is not a complete graph or an odd cycle,
by Theorems \ref{j1} and \ref{g3}, $\chi(G_1)\leq
\Delta(G_1)=\bigg[
\begin{array}{c}
n \\
\frac{n}{2}
\end{array}
\bigg]_{q}-q^{\frac{n^2}{4}}-1.$ So $\chi(G)\leq
\sum_{i=1}^{\frac{n}{2}-1} \bigg[
\begin{array}{c}
n \\
i
\end{array}
\bigg]_{q}+\bigg[
\begin{array}{c}
n \\
\frac{n}{2}
\end{array}
\bigg]_{q}-q^{\frac{n^2}{4}}-1$. Now the results follows by Lemma
\ref{g2}.
\end{proof}

\begin{corollary}
If $n$ is even then $$\sum_{i=1}^{\frac{n}{2}-1} \bigg[
\begin{array}{c}
n \\
i
\end{array}
\bigg]_{q}+\bigg[
\begin{array}{c}
n-1 \\
\frac{n-2}{2}
\end{array}
\bigg]_{q}\leq \chi(G(V))\leq \sum_{i=1}^{\frac{n}{2}-1} \bigg[
\begin{array}{c}
n \\
i
\end{array}
\bigg]_{q}+\bigg[
\begin{array}{c}
n \\
\frac{n}{2}
\end{array}
\bigg]_{q}-q^{\frac{n^2}{4}}-1.$$
\end{corollary}

\begin{theorem}
$\gamma(G(V))=q+1$.
\end{theorem}

\begin{proof} Let $W$ be a subspace of $V$ with
$dim(W)=n-2$. It follows that $\frac{V}{W}$ has $q+1$ subspaces
$\frac{W_1}{W},\frac{W_2}{W},...,\frac{W_{q+1}}{W}$ with
$dim(\frac{W_i}{W})=1$ for $i=1,2,...,q+1$. It is obvious that
$$\frac{V}{W}=\frac{W_1}{W}\cup\frac{W_2}{W}\cup...\cup\frac{W_{q+1}}{W}.$$
Now we can see that $$V=W_1\cup W_2\cup...\cup W_{q+1}.$$ This
means that $\{W_1,W_2,...,W_{q+1}\}$ is a dominating set for
$G(V)$ and so $\gamma(G(V))\leq q+1$. On the other hand suppose
that $S=\{V_1,V_2,...,V_t\}$ is a minimum dominating set for
$G(V)$. If there is a vector $x\not \in (V_1\cup V_2\cup ...\cup
V_t)$, then $\langle x \rangle$ is not dominated by $S$. This
contradiction implies that $V_1\cup V_2\cup ...\cup V_t=V$. Then
$$q^n=|\cup _{i=1}^{t} V_i| < \sum_{i=1}^{t}|V_i| \leq
\sum_{i=1}^{t}q^{n-1}  = tq^{n-1}.$$

We deduce that $t\geq q+1$, and so $\gamma(G(V))\geq q+1$. This
completes the proof.
\end{proof}

\begin{theorem}
$\alpha(G(V))=\frac{q^n -1}{q-1}$.
\end{theorem}

\begin{problem}
What is the exact value of $w(G(V))$ for even $n$?
\end{problem}

\end{document}